\documentclass[letterpaper, 10 pt, conference]{ieeeconf}  

\IEEEoverridecommandlockouts                              

\usepackage{graphicx}
\usepackage[caption=false,font=footnotesize]{subfig}
\usepackage{caption}
\graphicspath{
    {figures/}
}

\usepackage{amsmath,amssymb,enumerate}

\usepackage{amsthm}
\usepackage{setspace}
\usepackage{booktabs}
\usepackage[usenames,dvipsnames,svgnames,table]{xcolor}
\usepackage{mathtools}
\usepackage{algorithm, algorithmicx, algpseudocode}
\usepackage{blindtext}
\usepackage{gensymb}
\usepackage{xparse}
\usepackage{lipsum}
\usepackage{mathrsfs}
\usepackage[mathscr]{euscript}
\usepackage{times}
\usepackage{cite} 
\usepackage{multicol}
\definecolor{darkgreen}{rgb}{0,0.6,0}
\usepackage[bookmarks=true,colorlinks=true,pdfpagemode=UseNone,citecolor=darkgreen,linkcolor=black,urlcolor=BrickRed,pagebackref]{hyperref}
\usepackage[caption=false,font=footnotesize]{subfig}
\usepackage{amsfonts}
\usepackage[utf8]{inputenc}
\usepackage[T1]{fontenc}
\usepackage{textcomp}
\usepackage{arydshln}
\usepackage{balance}

\newtheorem{problem}{Problem}
\newtheorem{theorem}{Theorem}
\newtheorem{corollary}[theorem]{Corollary}
\newtheorem{proposition}[theorem]{Proposition}

\newtheorem{remark}{Remark}

\renewcommand{\thefigure}{\arabic{figure}}

\definecolor{note}{rgb}{0.1,0.1,1}
\definecolor{rephase}{rgb}{0.15,0.7,0.15}
\definecolor{bag}{rgb}{0.6,0.6,0.2}

\makeatletter
\renewcommand*\env@matrix[1][c]{\hskip -\arraycolsep
  \let\@ifnextchar\new@ifnextchar
  \array{*\c@MaxMatrixCols #1}}
\makeatother

\makeatletter
\newcommand{\mathleft}{\@fleqntrue\@mathmargin0pt}
\newcommand{\mathcenter}{\@fleqnfalse}
\makeatother

\definecolor{orange}{RGB}{255,127,0}

\title{\LARGE \bf
Learning Hybrid Dynamics via Convex Optimizations
}

\author{Kaito Iwasaki, Sangli Teng, Anthony Bloch, Maani Ghaffari
\thanks{Kaito Iwasaki, Sangli Teng, Anthony Bloch, Maani Ghaffari are with the University of Michigan, Ann Arbor, MI 48109, USA. \texttt{\{kaitoi, sanglit, abloch, maanigj\}@umich.edu}.
}
}

\begin{document}

\maketitle
\thispagestyle{empty}
\pagestyle{empty}

\begin{abstract}
This paper investigates the problem of identifying state-dependent switching systems, a class of hybrid dynamical systems that combine multiple linear or nonlinear modes. We propose two broad classes of switching systems: switching linear systems (SLSs) and switching polynomial systems (SPSs). We first formulate the joint estimation of the mode dynamics and switching rules as a mixed integer program. To solve its inherent scalability issue, we develop a hierarchy of convex relaxations and establish a bound and conditions under which these relaxations are tight. Building on these results, we propose a bilevel convex optimization framework that alternates between mode assignment and dynamics estimation, and we recover switching boundaries using margin-based polynomial classifiers. Numerical experiments on both linear and nonlinear oscillators demonstrate that the method accurately identifies mode dynamics and reconstructs switching surfaces from trajectory data. Our results provide a tractable optimization-based framework for switching system identification.
\end{abstract}

\IEEEpeerreviewmaketitle

\section{Introduction}
Switching systems constitute a class of hybrid dynamical systems that evolve according to one of several constituent vector fields, with a discrete mode selecting which dynamics are active at a given state, time, and/or input. This representation captures behaviors that are hard to describe with a single smooth model while preserving interpretability through simple mode-specific dynamics and explicit switching logics. Applications span robotics and autonomous driving, power and energy systems, and biological system.

The switching systems identification problem comprises:
\begin{enumerate}
\item \emph{Model estimation}: learn the parameters of each mode-specific model;
\item \emph{Switching logic estimation}: determine which mode is active at each state and time. 
\end{enumerate}
Each subproblem can be considered a classical system identification problem. For example, if the logic is known, the problem collapses to classical identification per mode. However, if both are unknown, conventional techniques may not be directly applicable due to their coupling, hence requiring a new approach.

\subsection{Related work}
There are several surveys \cite{GARULLI2012344, MORADVANDI202395} and background texts \cite{Liberzon2003swsystems, VidalBook2016, LauerBloch2019} on switching systems.
Research on switching system identification spans both \emph{modeling choices} and \emph{algorithmic paradigms}. On the modeling side, works differ by (i) the submodel class such as input-output (e.g., ARX \cite{DuClustering} or BJ \cite{PigaBayesian}), or state-space (linear \cite{SefidmazgiSMC2016}, affine \cite{DuTCSI2020}, nonlinear \cite{PaolettiEJC2007}); and (ii) how the switching rule is represented (state-/input-dependent partitions \cite{mavridis2024real}, time/event schedules \cite{Xiang2021}, or stochastic rules \cite{LIU20211136}).

Algorithmically, there are many variants loosely classified into the following four families. (1) \emph{Optimization based methods} aim for  a global fit with unknown mode labels such as mixed-integer program \cite{RollBemporadLjung2004} and greedy partitioning \cite{BemporadGreedyHSCC2003}. A substantial body of work within this category targets piecewise affine and piecewise ARX models, where the regressor space is partitioned into polyhedra with affine or ARX submodels. For instance, \cite{FerrariTrecateAutomatica2003} assumes a known number of subsystems and combine clustering, regression, and classification. (2) \emph{Algebraic approaches} transform SARX models into lifted linear relations that eliminate explicit dependence on the switching sequence. The connection between algebraic subspace clustering and ARX identification was first highlighted in \cite{VidalPhD2003}, with further development in recursive formulations in \cite{VidalAutomatica2008}. (3) Clustering/EM-type and prototype-based schemes use unsupervised assignments interleaved with per-mode estimation \cite{FerrariTrecateAutomatica2003,DuTCSI2020}. (4) \emph{Probabilistic formulations} treat the mode as latent, enabling uncertainty quantification and automatic model-order control \cite{PigaBayesian2020a}. Many methods appear in both offline and online/recursive forms, e.g., \cite{mavridis2024real}.

Within this landscape, our current work follows the state-space, state-dependent model and optimization-based algorithm (global fit with unknown modes), emphasizing tractable relaxations of mode constraints, and recovering switching boundaries with a margin-based polynomial nonlinear classifier, rather than fixing the linear or polyhedral partitions a prioi.

\section{Switching systems}\label{sec: switching system}
In this section, we formalize switching linear and polynomial systems, where the dynamics depend on the active mode selected by state-triggered switching. We first review the general notion of switching systems, then specialize to linear and polynomial vector fields.
\subsection{Switching Systems}
A \emph{switching system} is a dynamical system that evolves according to one of several vector fields, with the active mode changing based on state, time, or input conditions. Formally,
\begin{equation}\label{eq: switching system}
\dot{x}=X_q(x), \quad q \in \mathcal{M}=\left\{1, \ldots, M\right\},
\end{equation}
where each mode $q$ is associated with a smooth vector field $X_q$. Switching can be classified as state-dependent or time-dependent, and by more complex rules (e.g., hysteresis). In this paper we focus on state-dependent switching.

More explicitly, let $Q$ be the state space. Each mode $q$ is assigned a domain $C_q\subseteq Q$ where its dynamics evolve, and a guard set $D_q\subseteq Q$ where transitions may occur. Thus, the system can be written as
\begin{equation}
    \mathcal{S W}: \begin{cases}\dot{x}=X_q(x), & x \in C_q \\ q^{+} \in \mathcal{M}, & x \in D_q\end{cases},
\end{equation}
meaning the state flows under $X_q$ until reaching $D_q$, upon which the mode switches to some $q^+$. We assume non-Zeno behavior (nonzero time between switches). Switching systems can be viewed as a subclass of hybrid systems, where mode transitions may also include discrete state resets (see \cite{Goebel2008invariance,Liberzon2003swsystems,Goebel2009hybrid} for details).

\subsection{Switching Linear Systems in \texorpdfstring{$\mathbb{R}^n$}{TEXT}}
A \emph{switching linear system} (SLS) is a state-dependent switching system in which each mode is governed by a linear vector field. Let the mode set be $\mathcal{M}=\{1,2, \ldots, M\}$ and associate to each mode $j \in \mathcal{M}$ a constant system matrix $A_j \in \mathbb{R}^{n \times n}$. The continuous state $z(t) \in \mathbb{R}^n$ evolves according to
\begin{equation}\label{eq: SLS}
\dot{z}(t) = A_{\sigma(z(t))} z(t),
\end{equation}
where $\sigma: \mathbb{R}^n \rightarrow \mathcal{M}$ is the \emph{switching signal} specifying the active mode at time $t$. In the state-dependent setting considered here, the mode $\sigma(z(t))$ is determined by the position of $z(t)$ relative to a collection of \emph{switching surfaces} in the state space. These surfaces partition $\mathbb{R}^n$ into regions, each corresponding to a distinct mode.

A convenient way to express the mode assignment is through binary indicator variables 
\begin{equation}\label{eq: binary constraint}
    \lambda_j(z(t)) \in\{0,1\}, \quad \sum_{j=1}^M \lambda_j(z(t))=1,
\end{equation}
so that the dynamics \eqref{eq: SLS} can be written as
\begin{equation}\label{eq: SLS convex}
\dot{z}(t) = \sum_{j=1}^M \lambda_j(z(t)) A_j z(t).
\end{equation}
We remark that for our purposes, we later relax the discrete constraints on the $\lambda_j$ to allow continuous values in the interval $[0,1]$, while maintaining the simplex constraint, i.e.,
\begin{equation}\label{eq: lp mode constraint}
\lambda_j(z(t)) \in[0,1], \quad \sum_{j=1}^M \lambda_j(z(t))=1
\end{equation}
This convex relaxation enables the development of tractable optimization-based identification algorithms, which we explore in detail later. Note that in order to fully impose the binary constraint $\lambda_j \in \{0, 1\}$, we require in addition to \eqref{eq: lp mode constraint} a quadratic constraint on $\lambda_j$
\begin{equation}\label{eq: quadratic constraint}
    \lambda_j(z(t))(\lambda_j(z(t)) - 1) = 0.
\end{equation}
\subsection{Switching Polynomial Systems in \texorpdfstring{$\mathbb{R}^n$}{TEXT}}
A \emph{switching polynomial system} (SPS) is a state-dependent switching system in which each mode is governed by a polynomial vector field. Let the mode set be the same as SLS and associate to each mode $j \in \mathcal{M}$ a polynomial vector field $H_j: \mathbb{R}^n \rightarrow \mathbb{R}^n$. The state $z(t) \in \mathbb{R}^n$ evolves according to
\begin{equation}\label{eq: SPS}
\dot{z}(t) = H_{\sigma(z(t))}(z(t)),
\end{equation}
where the active mode $j=\sigma(z(t))$ is determined by the location of $z(t)$ relative to a set of switching surfaces, as in the linear case.

Each vector field can be expressed component-wise as $H_j(z)=(h_{j1}(z), \ldots, h_{jn}(z))^T$, where $h_{j k}(z) \in \mathbb{R}\left[z_1, \ldots, z_n\right]$ so that every coordinate function is a polynomial in the state variables, where we write the state vector as $z = (z_1, \cdots, z_n)^T$. 

Following the formulation for SLS, we introduce binary mode indicators as in \eqref{eq: binary constraint}. Again, as in SLS, we relax the binary constraints to \eqref{eq: lp mode constraint} enabling a convex representation of the mode combination. 
\section{Switching system identification as MIP}
Informally, the main problem we address in this paper is a method of identifying a switching system from data. At a high level, this switching system identification can be cast as a constrained minimization problem. 
\begin{problem}\label{prob: MIP}
    Suppose we have a set of measurements $\{(z^i,\dot z^{\,i})\}_{i=1}^N$ from a switching system with associated $\lambda$ and $X$. Consider a cost $c(\lambda, X| z^i, \dot z^i) \ge 0$ such that $c(\lambda, X| z^i, \dot z^i) = 0$ if $\dot z = \sum_{j=1}^N \lambda_j(z^i) X_j(z^i)$. We seek mode  assignments $\lambda^i := (\lambda_1(z^i), \ldots, \lambda_M(z^i)) \in \{0,1\}^M$ and mode dynamics $\{X_j\}_{j=1}^M$ that minimize the overall cost:
\begin{equation}\label{eq: MIP_form}
\begin{aligned}
    \min_{\lambda^i, X_j} \quad &
    \sum_{i=1}^N c(\lambda, X| z^i, \dot z^i) \\
    \text{s.t.}\quad & \lambda^i \in \{0,1\}^M, \quad \mathbf{1}^T \lambda^i = 1, \ \forall i
\end{aligned}
\end{equation}
\end{problem}

The main difficulty in solving this problem is that \eqref{eq: MIP_form} is a mixed integer program (MIP). Although it captures the identification task precisely, it is not scalable to a large data set as the search space grows exponentially as $N$ increases.

This motivates our study of the next step: \emph{convex relaxation}. By relaxing the binary constraints on $\lambda^i$ into convex constraints, we obtain tractable  convex optimization problems that approximate \eqref{eq: MIP_form}. We describe several methods of doing this in Section \ref{sec: convex relaxation} and give a new formulation of Problem \ref{prob: MIP} in Section \ref{sec: ss identification bilevel}.

\section{Convex relaxation of mode constraints}\label{sec: convex relaxation}
In preparation for reformulating Problem \ref{prob: MIP}, we formulate the mode selection rules into a convex optimization friendly form. We first introduce a method which is to replace the quadratic integrality conditions in \eqref{eq: quadratic constraint} with their lifted semidefinite counterparts, yielding a convex relaxation of the integer constraint. (e.g., via the Shor or moment relaxation). While such relaxations can offer tight bounds and, in some cases, exact recovery, their computational cost scales rapidly with the number of data points. To preserve scalability, we subsequently introduce another method which imposes only the linear simplex constraints in \eqref{eq: lp mode constraint}, which are subsequently "hardened" to satisfy the original single active mode constraint \eqref{eq: binary constraint}.
\subsection{Semidefinite relaxation}
Observe the mixed integer constraint \eqref{eq: binary constraint} is equivalent to 
\begin{equation}\label{eq: quad constraint}
    \begin{aligned}
        \lambda_j(1-\lambda_j) = 0 \quad \forall j, \quad \text{s.t.} \quad \mathbf{1}^T\lambda = 1,
    \end{aligned}
\end{equation}
which fits nicely into the moment relaxation framework. We will not go over the full details of the general theory of moment relaxation, but rather directly apply it to our special case of mode constraints. For readers interested, see \cite{Lasserre2001Moments}, \cite{Parrilo2003SDP}.

In the moment framework, we treat $\lambda$ as a polynomial variable and feasibility is enforced via a truncated moment sequence $y=\left(y_\alpha\right)_{|\alpha| \leq 2 r}$ indexed by monomials in $\lambda$.
The order-$r$ relaxation imposes
\begin{equation}\label{eq: moment condition}
\begin{aligned}
& M_r(y) \succeq 0, \\
& M_{r-1}((\lambda_j^2 - \lambda_j)y) = 0, \quad \forall j,\\
& M_{r-1}((\mathbf{1}^T \lambda - 1)y) = 0,
\end{aligned}
\end{equation}
where $M_r(y)$ is the moment matrix and $M_{r-d}(g y)$ the localizing matrix associated with constraint polynomial $g$.
When an objective is specified (e.g., from the system identification loss in our case), it is incorporated as some linear functional $L_y(\cdot)$ of the moment sequence.
For $r=1$, the moment conditions in \eqref{eq: moment condition} reduces to
\begin{equation}
    \begin{bmatrix}
    1 & \lambda^T \\
    \lambda & \Lambda
    \end{bmatrix}\succeq 0, \quad \Lambda_{j j}=\lambda_j, \quad \mathbf{1}^T \lambda=1
\end{equation}
where we used $\lambda$ to denote the first order moments $y_j$ for the purpose explained in Theorem \ref{thm: flat-extension}, and $\Lambda_{ij} = y_{ij}$. Note that this corresponds to the Shor's relaxation \cite{Shor1987Quadratic}. 

We call the order-$r$ moment relaxation \emph{tight} if there exists an optimal moment sequence $y^{\star}$ from which one can extract an optimal feasible $\lambda^{\star}$ for the original problem. the Lasserre hierarchy provides a nested sequence of such relaxations with nondecreasing optimal values, where $r=1$ coincides with the Shor PSD relaxation and higher orders yield tighter approximations \cite{Lasserre2001Moments}, \cite{Parrilo2003SDP}. To certify tightness at a fixed order, we invoke the Curto–Fialkow flat-extension criterion: a rank condition on the moment matrices forces $y^{\star}$ to be a Dirac measure on a canonical basis vector.
\begin{theorem}[Flat extension theorem \cite{CurtoFialkow1996}]\label{thm: flat-extension}
If for some $r \ge r_0$,
\begin{equation}\label{eq: rank condition}
\operatorname{rank} M_r(y^\star) = \operatorname{rank} M_{r-1}(y^\star) = 1,
\end{equation}
then $y^\star$ admits a representing measure supported at a single point $\lambda^\star \in \mathbb{R}^M$. By feasibility of the mode constraints, $\lambda^\star$ is a canonical basis vector, and the relaxation is tight.
\end{theorem}
\begin{proof}
By the Curto–Fialkow flat-extension theore, \eqref{eq: rank condition} implies that $y^\star$ has a representing measure with exactly one atom \cite{CurtoFialkow1996} \cite{CurtoFialkow1998}. Hence, $y^\star_\alpha = (\lambda^\star)^\alpha$ for some $\lambda^\star$. The constraints $\lambda_j^2-\lambda_j=0$ and $\mathbf{1}^T\lambda=1$ force $\lambda^\star \in \{e_1,\ldots,e_M\}$. See also \cite{Lasserre2001Moments} for the connection to moment relaxations of polynomial optimization.
\end{proof}
\begin{corollary}[Mode recovery from first-order moments]
Under the conditions of Theorem~\ref{thm: flat-extension},
\begin{equation}\label{eq: first order moment}
\lambda_j^\star = y_{e_j}^\star, 
\quad
M_1(y^\star) =
\begin{bmatrix}
1 & (\lambda^\star)^T\\[0.2em]
\lambda^\star & \lambda^\star (\lambda^\star)^T
\end{bmatrix}.
\end{equation}
Moreover, tightness is equivalent to $\operatorname{rank} M_r(y^\star) = 1$.
\end{corollary}
\begin{proof}
From the theorem, $y^\star$ is generated by the Dirac measure $\delta_{\lambda^\star}$, hence $y_\alpha^\star=(\lambda^\star)^\alpha$ and $\lambda_j^\star=y_{e_j}^\star$. Consequently,
\[
M_1(y^\star)
= \int \begin{bmatrix}1\\ \lambda\end{bmatrix}\begin{bmatrix}1\\ \lambda\end{bmatrix}^{T}d\delta_{\lambda^\star}
= \begin{bmatrix}1\\ \lambda^\star\end{bmatrix}\begin{bmatrix}1\\ \lambda^\star\end{bmatrix}^T,
\]
which yields the right hand side of \eqref{eq: first order moment} and is rank-one. Conversely, if the relaxation is tight and returns a feasible $\lambda^\star$, the induced moment sequence from $\delta_{\lambda^\star}$ yields $\operatorname{rank} M_r(y^\star)=1$.
\end{proof}

\subsection{Further relaxation to simplex constraints}\label{subsec: LP-relax}
While the SDP relaxation encodes the quadratic integrality condition, its scalability is still limited for large datasets. As a scalable alternative, we drop integrality and keep only the linear simplex constraint
\begin{equation}
\Delta:=\{\lambda\in\mathbb{R}^M:\ \lambda\ge 0,\ \mathbf{1}^T\lambda=1\}.
\end{equation}
For a sample $(x,\dot x)$, write $v_j:=X_j(x)$. For the cost $c$ we choose a generic convex positively homogeneous loss $\mathcal{L}:\mathbb{R}^n\to\mathbb{R}_+$ (e.g., any $\ell_p$-norm):
\begin{equation}
    c(\lambda, X|x,\dot x) := \ \mathcal{L}\left(\dot x-\sum_{j=1}^{M}\lambda_j v_j\right).
\end{equation}
Then, the relaxed mode selection problem is
\begin{equation}\label{eq:generic-relax}
\min_{\lambda\in\Delta}\ \mathcal{L}\left(\dot x-\sum_{j=1}^M\lambda_j v_j\right).
\end{equation}
\begin{remark}
    We will specialize to $\mathcal{L}(r)=\|r\|_p$ (in particular, $p = 1,2$) in Section \ref{sec: ss identification bilevel}, which yields a LP and a SDP.
\end{remark}

We state the dual certificate for a one-hot optimum in the following. Here \emph{one-hot} means $\lambda^\star=e_{j^\star}$ (a vector with a single 1 and all other entries 0). Let $\mathcal{L}^\ast$ be the dual (dual norm if $\mathcal{L}$ is a norm), and recall the representation
$\mathcal{L}(r)=\sup_{\mathcal{L}^\ast(w)\le 1}\langle w,r\rangle$.
\begin{theorem}[One–hot optimality under a dual certificate]
If there exist $j^\star$ and $w$ with $\mathcal{L}^\ast(w)\le 1$ such that
$w\in \partial \mathcal{L}(\dot x-v_{j^\star})$ and
\[
\langle w,v_{j^\star}\rangle\ >\ \max_{j\ne j^\star}\langle w,v_j\rangle,
\]
then the unique minimizer of \eqref{eq:generic-relax} is $\lambda^\star=e_{j^\star}$.
\end{theorem}
\begin{proof}
Using $\mathcal{L}(r)=\sup_{\mathcal{L}^\ast(w)\le 1}\langle w,r\rangle$ and the minimax theorem,
\[
\min_{\lambda\in\Delta}\mathcal{L}\big(\dot x-\sum_j\lambda_j v_j\big)
=\sup_{\mathcal{L}^\ast(w)\le 1}\ \left(\langle w,\dot x\rangle-\max_j\langle w,v_j\rangle\right).
\]
If $w\in\partial \mathcal{L}(\dot x-v_{j^\star})$ and $j^\star$ uniquely maximizes $\langle w,v_j\rangle$, KKT conditions force $\lambda^\star=e_{j^\star}$.
\end{proof}
\begin{theorem}
    Let $v_i\neq v_j$ for $i\neq j$. If $\dot x$ is drawn from a distribution absolutely continuous w.r.t. Lebesgue measure and $\mathcal{L}$ is a norm, then
\begin{equation}
\mathbb{P}\left(\exists\,i\neq j:\ \mathcal{L}(\dot x-v_i)=\mathcal{L}(\dot x-v_j)\right)=0,
\end{equation}
so $j^\star\in\arg\min_j \mathcal{L}(\dot x-v_j)$ is unique a.s., and any optimizer of \eqref{eq:generic-relax} is one–hot.
\end{theorem}
\begin{proof}
    For each pair $(i, j)$, the equality $\mathcal{L}(\dot x-v_i)=\mathcal{L}(\dot x-v_j)$ holds only on a finite union of affine hyperplanes as it is the zero level of a nonconstant Lipschitz function. Hence, this set of such tie points is measure zero. The finite union over all such pairs $(i, j)$ still remains measure zero.
\end{proof}
The above results ensure that, under generic conditions, the optimizers $\lambda^{\star}$ are almost always binary. In practice, however, we may lack exact knowledge of $X_j$ or have noisy measurements of $(x, \dot{x})$, leading to fractional "soft" solutions $\lambda^\star$. A common post-processing step is to "harden" such solutions by setting the largest component $\lambda_{j^{\star}}^{\star}$ to 1 and all the others to 0, thereby recovering the binary constraint \eqref{eq: binary constraint}. While this rounding is certainly not guaranteed to preserve optimality, the next result shows that it is safe: the increase in residual error is bounded explicitly in terms of how concentrated $\hat{\lambda}$ is and how far the mode vectors $v_j$ are from the selected mode $v_{j^{\star}}$.
\begin{proposition}
For any $\hat\lambda\in\Delta$ and any $j^\star\in\arg\max_j \hat\lambda_j$,
\begin{equation}\label{eq: any-norm-rounding}
\mathcal{L}(\dot x-v_{j^\star}) \le \mathcal{L}\left(\dot x-\sum_{j=1}^M\hat\lambda_j v_j\right) +\sum_{j=1}^M\hat\lambda_j\mathcal{L}\left(v_j-v_{j^\star}\right),
\end{equation}
Thus, rounding a soft solution to the largest component increases the loss by at most the (mode–weighted) dispersion around $v_{j^\star}$.
\end{proposition}
\begin{proof}
    Apply the triangle inequality to
    \[
        \dot{x}-v_{j^{\star}}=\left(\dot{x}-\sum_{j=1}^{M} \hat{\lambda}_j v_j\right)+\sum_{j=1}^{M} \hat{\lambda}_j\left(v_j-v_{j^{\star}}\right). \qedhere
    \]
\end{proof}

In summary, the simplex relaxation \eqref{eq:generic-relax} preserves discrete behavior under mild certificates and is generically one–hot. When it is not, the hardening error is explicitly controlled by \eqref{eq: any-norm-rounding}. In switching system identification, we are interested in recovering the vector fields and switching rules simultaneously. In the next section, we outline the sequential convex optimizations similar to \cite{ACC2025} to find exact structure of the certain classes of switching systems.
\section{Bilevel convex optimizations}\label{sec: ss identification bilevel}
We return to Problem \ref{prob: MIP} and reformulate the MIP into \emph{bilevel convex optimizations} using the relaxation in Section \ref{sec: convex relaxation}. Due to page constraints, we present the case where we use the first-order moment relaxation. The higher-order cases and the linear simplex relaxation work analogously.
\subsection{Problem formulation}
We consider \emph{linear-in-parameters} models $F_j(z)=C_j \Phi(z)$ with a feature map $\Phi:\mathbb{R}^n\to\mathbb{R}^P$:
\begin{equation}\label{eq:lin_in_params}
    F_j(z) = C_j \Phi(z), \quad C_j \in \mathbb{R}^{n \times P}.
\end{equation}
In particular, in this work we consider the two important specializations introduced in Section \ref{sec: switching system}, i.e.,
\begin{itemize}
\item \textbf{SLS}: $\Phi_{\text{lin}}(z)=z$ and $C_j=A_j\in\mathbb{R}^{n\times n}$;
\item \textbf{SPS}: $\Phi_{\text{poly}}(z)=\phi_d(z)$ is a vector of monomials up to degree $d$, and each $C_j \in \mathbb{R}^{n\times P}$ is a matrix of polynomial coefficients.
\end{itemize}
Given the sampled data, our goal is to estimate $\{C_j\}_{j=1}^M$ and $\{\lambda^i\}_{i=1}^N$ by solving the $p$-norm minimization problem:
\begin{problem}\label{prob: ell1 norm SDP}
We minimize the discrepancy between the measured time derivative and the fitted switching system
\begin{equation}\label{eq: 1-norn with sdp constraints}
    \begin{aligned}
    \min_{\lambda^i,\Lambda^i,C_j}  
    &\  \sum_{i=1}^N  \left\|\dot{z}_i - \sum_{j=1}^M \lambda_j^i  F_j(z^i) \right\|_p \\
    \mathrm{s.t.}\quad &\ \lambda_j^i(1-\lambda_j^i) = 0, \quad \mathbf{1}^T\lambda^i = 1, \ \forall i, j,\\
    & \quad-\eta \le C_j \le \eta, \quad \forall j,
    \end{aligned}
\end{equation}
where $\eta \in \mathbb{R}$ defines a box constraint to ensure the feasible set is compact.
\end{problem}

Note that the $\ell_p$-cost in Problem \ref{prob: ell1 norm SDP} is nonconvex as it involves cross terms between decision variables $\lambda^i$ and $C_j$. There are several ways to handle this problem (e.g., for $p =2$, we can apply the moment relaxation to the corresponding polynomial optimization problem and turn it into a SDP). Our approach instead is to use $\ell_1$-norm, and turn it into an alternating subproblems by fixing either $\lambda^i$ or $C_j$ and introducing slack variables.
\subsection{Bilevel convex program}
We minimize the $\ell_1$ data mismatch and alternate between mode assignment and dynamics estimation. Introduce slack variables $\delta_k^i \ge 0$ for $k=1,\dots,n$. Let us denote $c^i_k(C,\lambda) := [\dot z^{i} - \sum_{j=1}^M \lambda_j^i C_j \Phi(z^i)]_k$. Note that when either $C$ or $\lambda$ is fixed, this will be linear in the free variables. We first fix the coefficients $C_j$ of the vector fields. We then apply the order-1 moment relaxation and solve the following SDP:
\begin{problem}[Mode assignment as SDP]\label{prob: mode search SDP}
Given fixed $\{C_j\}_{j=1}^M$, solve for $\{\lambda^i,\Lambda^i\}_{i=1}^N$
\begin{equation}\label{eq:mode_SDP_general}
\begin{aligned}
\min_{\lambda^i,\Lambda^i,\delta^i}\quad & \sum_{i=1}^N\sum_{k=1}^n \delta_k^i \\
\text{s.t.}\quad
& -\delta_k^i \le c^i_k(C,\lambda) \le \delta^i_k, \quad \delta^i_k \ge 0, \\
&\begin{bmatrix}
1 & (\lambda^i)^T \\
\lambda^i & \Lambda^i
\end{bmatrix} \succeq 0, \quad \operatorname{diag}(\Lambda^i) = \lambda^i,\quad \mathbf{1}^T \lambda^i = 1.
\end{aligned}
\end{equation}
\end{problem}
This is the order-$1$ moment (Shor) relaxation of the mixed-integer constraints on $\lambda^i$. In practice, we recover the mode either from $\lambda^i$ directly or by hardening to $e_{j^\star}$ with $j^\star \in \arg\max_j \lambda_j^i$. Note that for relaxations of order $r > 1$, the inequality constraints for the slack variables will become LMIs for the corresponding localizing matrices similar to \eqref{eq: moment condition}. With these fixed $\lambda^i$, we solve the following LP:
\begin{problem}[Dynamics estimation as LP]
    Given fixed $\{\lambda^i\}_{i=1}^N$, solve for $\{C_j\}_{j=1}^M$
\begin{equation}\label{eq:dynamics_LP_general}
\begin{aligned}
\min_{C_j,\delta^i}\ & \sum_{i=1}^N \sum_{k=1}^n \delta_k^i\\
\text{s.t.}\ 
&-\delta_k^i \le c^i_k, \le \delta^i_k,\ \delta^i_k \ge 0, \ -\eta \le (C_j)_{\ell m} \le \eta.
\end{aligned}
\end{equation}
\end{problem}
We alternate these subproblems until convergence of the mismatch. We make the following observation. Let
\begin{equation}
F(C,\lambda) := \sum_{i=1}^N \left\|\dot z^{\,i} - \sum_{j=1}^M \lambda_j^i C_j \Phi(z^i)\right\|_1,
\end{equation}
with block variables $C := \{C_j\}_{j=1}^M$ and the collection of first-order moments $\lambda := \{\lambda^i\}_{i=1}^N$. 
Although the moment relaxation introduces auxiliary second-order matrices $\Lambda^i$, the objective depends only on $\lambda^i$. The feasible sets are
\begin{equation}
\begin{aligned}
\mathcal{C} &:= \Big\{ C : -\eta \le (C_j)_{\ell m} \le \eta \ \ \forall j,\ell,m \Big\}, \quad \mathcal{M} := \prod_{i=1}^N \mathcal{M}_i,
\end{aligned}
\end{equation}
where each per-sample feasible set $\mathcal{M}_i$ is a collection of pairs $(\lambda^i, \Lambda^i)$ satisfying the constraints in \eqref{eq: 1-norn with sdp constraints}. Each block subproblem is convex. For fixed $C$, minimizing $F(C,\lambda)$ over $\mathcal{M}$ (with the usual slack variables for the $\ell_1$ norm) is an SDP, and for fixed $\lambda$, minimizing $F(C,\lambda)$ over $\mathcal{C}$ is an LP.
\begin{theorem}[Monotone decrease of cost]\label{thm: algorithm converges}
Let $\{(C^{(k)}, \Lambda^{(k)})\}_{k=1}^\infty$ be the sequence generated by the above alternating minimization:
\begin{equation}
\begin{aligned}
     \Lambda^{(k+1)} &\in \arg \min _{\Lambda \in \mathcal{L}} F\left(C^{(k)}, \Lambda\right),\\ C^{(k+1)} &\in \arg \min _{C \in \mathcal{C}} F\left(C, \Lambda^{(k+1)}\right) .
\end{aligned}
\end{equation}
Then, the objective values $F^{(k)}:=F\left(C^{(k)}, \Lambda^{(k)}\right)$ form a nonincreasing sequence bounded below, hence $F^{(k)} \rightarrow F^{\star}$.
\end{theorem}
\begin{proof}
    By optimality of each block step,
    \[F\left(C^{(k)}, \Lambda^{(k+1)}\right) \leq F\left(C^{(k)}, \Lambda^{(k)}\right),\]
    and
    \[F\left(C^{(k+1)}, \Lambda^{(k+1)}\right) \leq F\left(C^{(k)}, \Lambda^{(k+1)}\right) .\]
    Then $F\left(C^{(k)}, \Lambda^{(k)}\right)$ is nonincreasing thus converges to some $F^{\star}$ since $F$ is bounded below by $0$. 
\end{proof}
Theorem \ref{thm: algorithm converges} guarantees a monotone, convergent sequence of objective values, but this alone does not describe where the iterates may accumulate. To rule out cycling on a level set and to characterize admissible limits, we next show that every accumulation point is blockwise optimal (i.e., a fixed point of the two convex subproblems).
\begin{proposition}
    Let $(\bar{C}, \bar{\Lambda})$ be an accumulation point of $\{(C^{(k)}, \Lambda^{(k)})\}$. Then
    \begin{equation}
         \bar{\Lambda} \in \arg \min _{\Lambda \in \mathcal{L}} F(\bar{C}, \Lambda), \quad \bar{C} \in \arg \min _{C \in \mathcal{C}} F(C, \bar{\Lambda}).
    \end{equation}
\end{proposition}
\begin{proof}
    Take a convergent subsequence $C^{\left(k_t\right)} \to \bar{C},\,\Lambda^{(k_t)} \to\bar{\Lambda}$. By definition of the updates, $\Lambda^{(k_t+1)}$ minimizes $F\left(C^{\left(k_t\right)}, \cdot\right)$. Closedness of $\arg \min$ for convex LPs and continuity of $F$ yield $\bar{\Lambda} \in \arg \min _{\Lambda} F(\bar{C}, \Lambda)$. The other block is analogous.
\end{proof}
\subsection{Recovery of Switching Surfaces}
The identification step above produces pointwise mode assignments (hardened $\lambda^i$) at samples $\{z^i\}_{i=1}^N$, but not global switching sets. We therefore fit switching surfaces as zero level sets of polynomials
\begin{equation}\label{eq: monomial representation}
f_\ell(x)=a_\ell^T \phi_d(x),\quad a_\ell\in\mathbb{R}^P.
\end{equation}
We encode all $M$ modes via $L$ polynomial surfaces whose signs partition $\mathbb{R}^n$ into $2^L$ regions. Choose a ``mode-book'' $\{s_j\}_{j=1}^M\subset\{-1,+1\}^L$ (i.e., each $s_j$ is a vector of length $L$ with coefficients $\pm1$, and $L$ is minimal with $M\le 2^L$). Define mode regions as
\begin{equation}\label{eq: mode region}
\mathcal{R}_j:=\left\{x: \operatorname{sign}(f_\ell(x))=(s_j)_\ell, \ \forall \ell\right\}.
\end{equation}

We then recover the surfaces via the following LP:
\begin{problem}\label{prob: surface fitting}
To each sampled state $z^i$ with identified $\lambda^i$, we associate a mode $\alpha_i \in \{1,\ldots,M\}$ such that $z^i \in \mathcal{R}_{\alpha_i}$. Set $\sigma_{i,\ell}:=(s_{\alpha_i})_\ell$ and solve
\begin{equation}\label{eq: surface recovery}
\begin{aligned}
\min_{a_\ell,\xi_{i,\ell}} \quad & \sum_{i=1}^N\sum_{\ell=1}^L \xi_{i,\ell}+\beta\sum_{\ell=1}^L \zeta_{\ell,m}\\
\text{s.t.}\quad & \sigma_{i,\ell}\, f_\ell(z^i) \ge \varepsilon-\xi_{i,\ell},\quad \forall i, \ell,\\
&\xi_{i,\ell}\ge 0, \quad \forall i, \ell,\\
&\zeta_{\ell,m} \ge (a_\ell)_m, \quad\zeta_{\ell,m} \ge -(a_\ell)_m,\\
-&\eta \le (a_\ell)_m \le \eta,\quad \forall\,i,\ell,m,
\end{aligned}
\end{equation}
with margin $\varepsilon>0$, sparsity $\beta\ge 0$, and box bound $\eta>0$.
\end{problem}

Notice that this program is directly analogous to a soft-margin SVM. The margin constraint 
enforces correct sign separation up to slack $\xi_{i,\ell}$, while the 
$\ell_1$ penalty $\beta \sum_\ell \|a_\ell\|_1$ promotes less complicated polynomials. The main difference is that instead of linear hyperplanes, we use linear-in-parameters polynomial features $\phi_d(x)$ to approximte nonlinear switching boundaries.
\begin{theorem}[Existence of a margin $\varepsilon>0$] \label{thm: eps exists}
Suppose there exist coefficients $\{\hat a_\ell\}_{\ell=1}^L$ with $|(\hat a_\ell)_m|\le \eta$ and some  $t>0$ such that 
\begin{equation}
\sigma_{i,\ell}\,\hat a_\ell^T \phi_d(z^i) \ge t,\quad \forall i=1,\ldots,N,\, \ell=1,\ldots,L .
\end{equation}
Let $S:=\sum_{\ell=1}^L\left\|\hat{a}_{\ell}\right\|_1$. If $\beta S<N L t$, then any choice of $\varepsilon$ satisfying
\begin{equation}
\frac{\beta S}{N L}<\varepsilon \leq t
\end{equation}
guarantees that Problem \ref{prob: surface fitting} admits an optimal solution with $\xi_{i, \ell}=0$ for all $i, \ell$, and $a_{\ell} \neq 0$.
\end{theorem}
\begin{proof} 
Since $\varepsilon\le t$ and $\sigma_{i,\ell}\,\hat a_\ell^T \phi_d(z^i)\ge \varepsilon$, the point $(\hat a_\ell,\xi_{i,\ell}= 0)$ is feasible. Thus, the best objective value by $\hat a_\ell$ is $\beta S$. Suppose we choose $a_\ell= 0$, then $\xi_{i,\ell}\ge \varepsilon$. Hence, the best solution in this case has objective value $NL\varepsilon$. Since $\varepsilon>\beta S/N L$, we have $\beta S < N L\varepsilon$. Therefore, the feasible point $(\hat a_\ell, \xi_{i,\ell} = 0)$ strictly improves the zero solution, so any optimizer $a^\star_\ell$ must be nonzero.
\end{proof}
\begin{remark}
    Note that we can obtain such a certificate $\hat a_\ell$ and margin $t > 0$ by solving the following LP
    \begin{equation}
        \max _{t,a_{\ell}}\ t \quad \text {s.t.} \quad \sigma_{i, \ell} a_{\ell}^T \phi_d(z^i) \geq t, \quad |(a_{\ell})_m| \leq \eta.
    \end{equation}
    Moreover, if $\beta=0$, any $\varepsilon \in(0, t]$ prevents the trivial solution and yields zero slacks with a nonzero $a_\ell$.
\end{remark}
\section{Examples}
We provide numerical examples, one for SLSs and another for SPSs, using MOSEK \cite{mosek2025} to solve convex optimizations. We compare the first-order moment and the simplex relaxations for the mode identification step in problem \ref{prob: mode search SDP}.
\subsection{SLS identification: Switching damped oscillator}
Consider a damped harmonic oscillator with its state vector $z = (x, \dot{x})^T\in \mathbb{R}^2$. The state follows the dynamics
    \[\dot{z}(t) = \begin{bmatrix}
        \dot{x}(t)\\ \ddot{x}(t)
    \end{bmatrix} = \begin{bmatrix}
        0 & 1\\ -\omega^2 & -\gamma
    \end{bmatrix}\begin{bmatrix}
        x(t) \\ \dot{x}(t)
    \end{bmatrix} = Az(t),\]
    where $\omega$ is its angular frequency and $\gamma$ is the damping coefficient. Suppose that the surface characteristic changes depending on the position of the oscillator. We can model this by assigning different damping coefficients depending on the position, say across the line $L =\{z \in \mathbb{R}^2: x = 0\}$, so the system undergoes switching. Now, define the SLS with 
    \[A_1 = \begin{bmatrix}
        0 & 1\\-\omega^2 & -\gamma_1
    \end{bmatrix}, \quad A_2 = \begin{bmatrix}
        0 & 1\\ -\omega^2 & -\gamma_2
    \end{bmatrix}\]
    where $\gamma_1$ and $\gamma_2$ are distinct damping coefficients. 

    The coefficients in this example are $\omega = 1$, $\gamma_1 = 0.1, \gamma_2 = 0.5$. The initial guesses for $A_1$ and $A_2$ are both identity matrices, and the batch size is $2000$. Figure \ref{fig: SLS LP SDP comparison} shows the convergence of the algorithm, the number of mode mismatch at each iteration, and the proportion of moment matrices satisfying the Curto-Fialkow rank-1 condition, which increases steadily, indicating progressively tighter relaxations and near-integral mode assignments. Over 10 iterations, the moment relaxation required a total of $699.99 s$ (median $70.95 s$/iteration), whereas the simplex relaxation required $59.34 s$ in total (median $5.88 s$/iteration), highlighting the order-of-magnitude runtime advantage of LP.
    \begin{figure}[!ht]
    \centering
    \small
    \setlength{\tabcolsep}{0pt}
    \begin{tabular}{@{}cc@{}}
    \includegraphics[width=.5\linewidth]{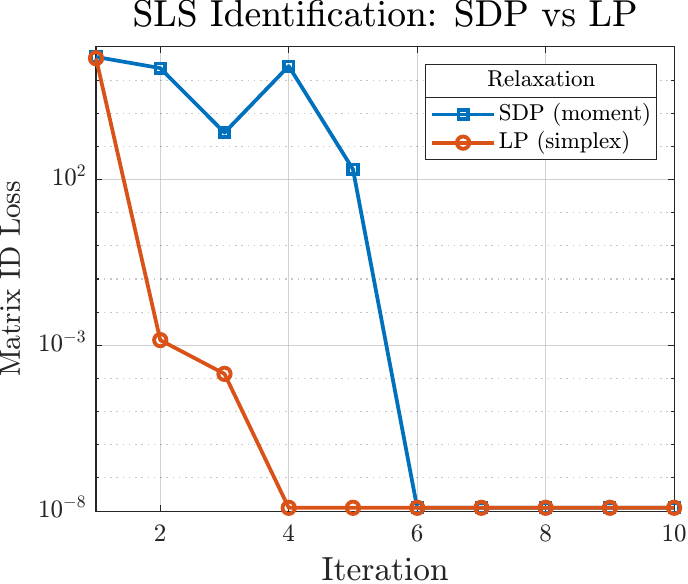} &
    \includegraphics[width=.5\linewidth]{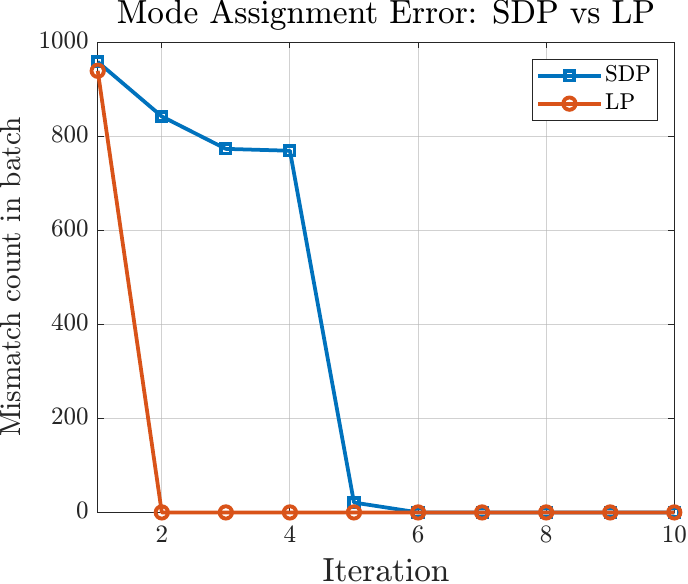} \\
    \includegraphics[width=.45\linewidth]{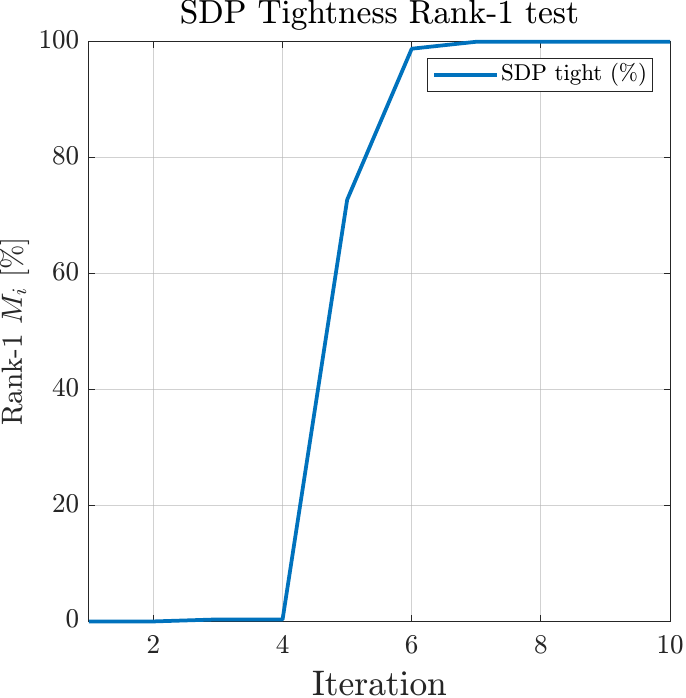} &
    \includegraphics[width=.45\linewidth]{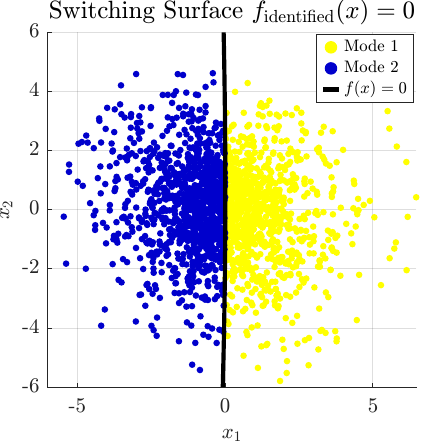}
    \end{tabular}
    \caption{Comparison of convex relaxations for switching linear system identification. Top left: identification cost for moment and simplex relaxations. Top right: mode assignment mismatch. Bottom left: SDP tightness ratio across relaxations. Bottom right: recovered switching surface.}
    \label{fig: SLS LP SDP comparison}
    \end{figure}
    After solving for $\lambda_i$, we solve the LP in Problem \ref{prob: surface fitting} to recover the switching surface. As in Theorem \ref{thm: eps exists}, we chose the parameters $d = 2$, $\eta = 10$, $\varepsilon = 10^{-2}$, $\beta = 10^{-2}$. The identified polynomial
    \[f_{\text{identified}}(x,y) = -0.0001 + 2.3837x + 0.0026y + 0.0026x^2\]
    matches well with the correct polynomial $f_{\text{correct}}(x,y) = x$ (the zero level set $f^{-1}(0)$ is unchanged up to scalar multiples) as in Fig~\ref{fig: SLS LP SDP comparison}, which shows the recovered mode labels at each sampled point and the plot of the identified surface. 
    \begin{figure}[!ht]
    \centering
    \small
    \setlength{\tabcolsep}{0pt}
    \begin{tabular}{@{}cc@{}}
    \includegraphics[width=.48\linewidth]{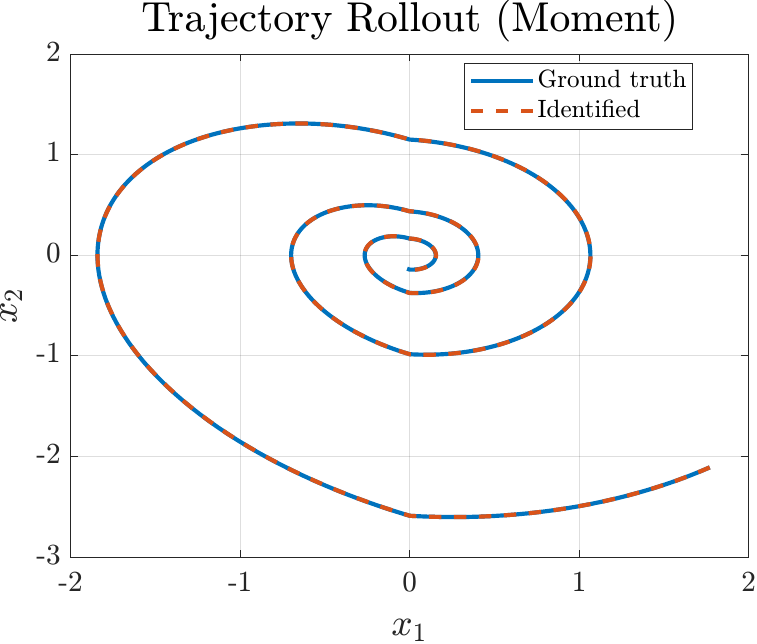} &
    \includegraphics[width=.48\linewidth]{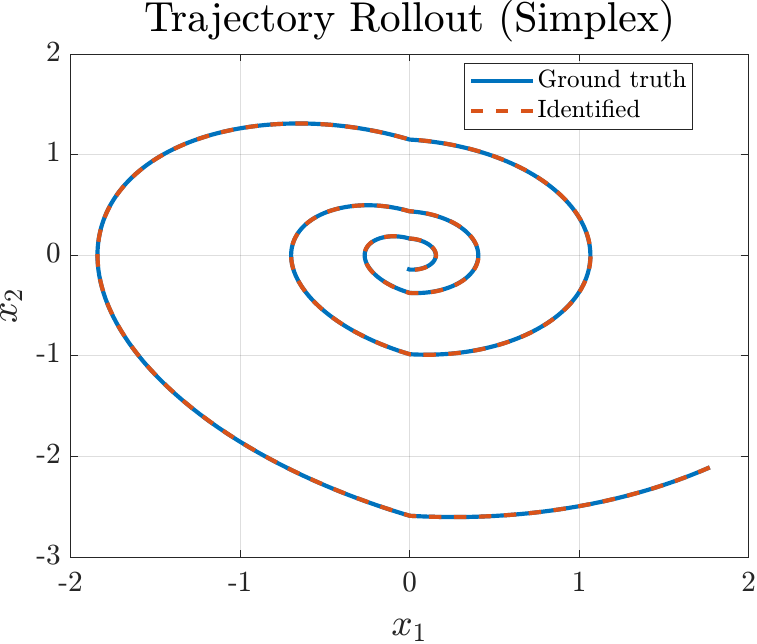} \\
    \includegraphics[width=.48\linewidth]{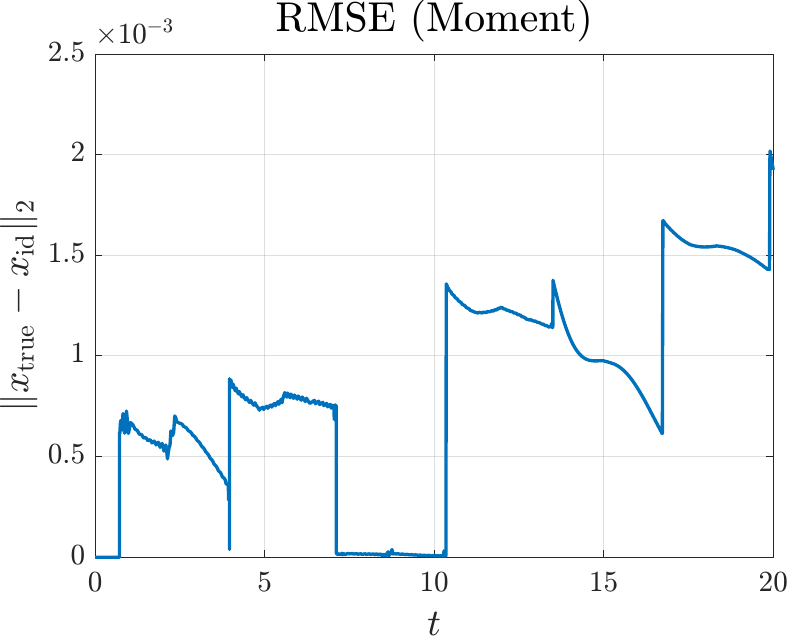} &
    \includegraphics[width=.48\linewidth]{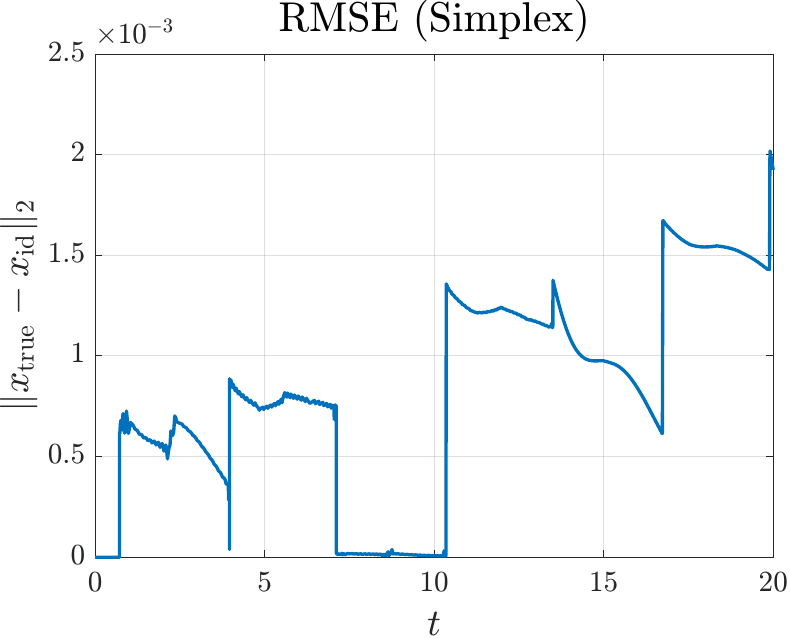}
    \end{tabular}
    \caption{Trajectory rollouts (top) and error evolution (bottom) for the identified switching linear system. Left column: moment relaxation. Right column: simplex relaxation. The rollout plots compare identified trajectories (solid) using the identified switching system with ground truth (dashed), and the error plots show the corresponding RMSE over time.}
    \label{fig: SLS LP SDP rollout RMSE}
    \end{figure}
    
    We evaluate the identified dynamics in two complementary ways. First, we assess \emph{pointwise} accuracy on a disjoint (from the training set) test set of state-velocity pairs, measuring the regression error (RMSE) and classification quality (accuracy and mIoU) of the learned switching surfaces. Second, we assess \emph{trajectory-level} accuracy by simulating the identified models on test initial conditions and comparing against ground-truth rollouts. Figure~\ref{fig: SLS LP SDP rollout RMSE} shows representative rollouts and RMSE profiles over time, while Table~\ref{tab: SLS_metrics} reports quantitative metrics aggregated over both the static test set and dynamic rollouts. Rollout RMSE measures the average trajectory discrepancy over time across $12$ different trajectories, while final and max error capture the deviation at the endpoint and the worst-case deviation along the trajectory.
    \begin{table}[!ht]
      \centering
      \renewcommand{\arraystretch}{0.8}
      {%
        \small
        \begin{tabular}{lrr}
          \toprule
          \textbf{Metric (pointwise test)} & \textbf{LP} & \textbf{SDP} \\
          \midrule
          Velocity RMSE & 0.0394 & 0.0394 \\
          Mode Accuracy & 0.9995 & 0.9995 \\
          mIoU          & 0.9990 & 0.9990 \\
          \midrule
          \textbf{Metric (trajectory rollout)} & & \\
          \midrule
          Rollout RMSE  & 0.0010 & 0.0010 \\
          Final error   & 0.0019 & 0.0019 \\
          Max error     & 0.0020 & 0.0020 \\
          \bottomrule
        \end{tabular}%
      }
    \caption{Quantitative evaluation of identified switching linear systems. 
    Top: pointwise regression and classification metrics on disjoint test data. 
    Bottom: trajectory-level errors and mode recovery quality on simulated rollouts.}
    \label{tab: SLS_metrics}
    \end{table}
\subsection{SPS identification: Switching quartic oscillator}
Consider a quartic potential of the form (see e.g., \cite{PhysRevD.110.123511})
    \begin{equation}
    V(x) = \frac{\omega^2}{2}x^2 + \frac{\lambda}{4}x^4,
    \end{equation}
    which gives an equation of a quartic oscillator.
    The energy of the quartic oscillator system is defined as
    \begin{equation}
        E(x, \dot x) = \frac{1}{2}\dot x^2 + \frac{\omega^2}{2}x^2 + \frac{\lambda}{4} x^4.
    \end{equation}
    We introduce nonlinear damping $D$ and forcing $F$ terms to the quartic oscillator such that the dynamics becomes
    \begin{equation}
    \ddot{x} + \omega^2 x + \lambda x^3 + D(x,\dot{x}) = F(x,\dot{x}).
    \end{equation}
    Let us suppose that $D$ and $F$ are both polynomials (or approximated by some polynomials) up to 4th order. The state $z = (x, \dot{x})^T$ follows a nonlinear system
    \begin{equation}
    \dot{z}(t) = \begin{bmatrix}
        \dot{x}\\ \ddot{x} 
    \end{bmatrix} = \begin{bmatrix}
        \dot{x}\\ -\omega^2 x - \lambda x^3 - D(x,\dot{x}) + F(x,\dot{x})
    \end{bmatrix}.
    \end{equation}
    Denote $y := \dot{x}$, then to identify the above dynamics, we model it as a polynomial system
    \begin{equation}
        \dot{z} = \begin{bmatrix}
            \dot{x} \\ \dot y
        \end{bmatrix} = \begin{bmatrix}
            a^T \phi_d(x,y)\\ b^T\phi_d(x,y)
        \end{bmatrix} = H(x,y),
    \end{equation}
    where $\phi_d(x,y)$ is the monomial basis as introduced in \eqref{eq: monomial representation}, and $a, b \in \mathbb{R}^P$ are the coefficient vectors. Suppose that the dynamics changes when the system crosses some energy threshold $E_c$, that is, $\dot z = H_1(x,y)$ if $E(x,y) \le E_c$ and $\dot z =H_2(x,y)$ if $E(x,y) > E_c$, where each mode has a different polynomial system.
    In our numerical example, we chose the parameters of the system to be $\omega = 1$ and $\lambda = 0.3$. We chose the damping and the forcing terms in each mode so that the polynomial vector fields have the form
    \begin{equation}
    \begin{aligned}
        H_1(x,y) &= \begin{bmatrix}
            y\\ -\omega^2 x - \lambda x^3 - 0.1 y^3 - 0.1 + 0.2x^2 + 0.3xy
        \end{bmatrix},\\
        H_2(x,y) &= \begin{bmatrix}
            y\\ -\omega^2 x - \lambda x^3 - 0.1y^3 + 0.1xy + 0.1x^2 - 0.1y 
        \end{bmatrix}.
    \end{aligned}
    \end{equation}
    The energy threshold is chosen to be $E_c = 3.0$, making the switching surface defined as $S = f^{-1}(0)$ by the polynomial
    \begin{equation}
    f(x,y) = E(x,y) - E_c = \frac{1}{2}y^2 + \frac{\omega^2}{2}x^2 + \frac{\lambda}{4}x^4 - 3.0
    \end{equation}
    Figure \ref{fig: SLS LP SDP comparison} shows the convergence of the algorithm for this switching polynomial system identification and the number of mode mismatches with batch size $2000$.
    \begin{figure}[!ht]
    \centering
    \small
    \setlength{\tabcolsep}{0pt}
    \begin{tabular}{@{}cc@{}}
    \includegraphics[width=.5\linewidth]{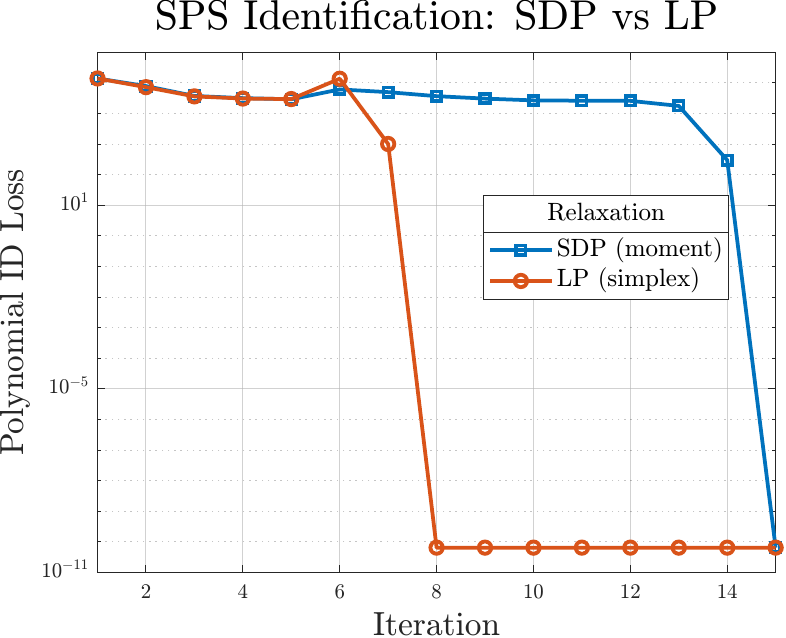} &
    \includegraphics[width=.5\linewidth]{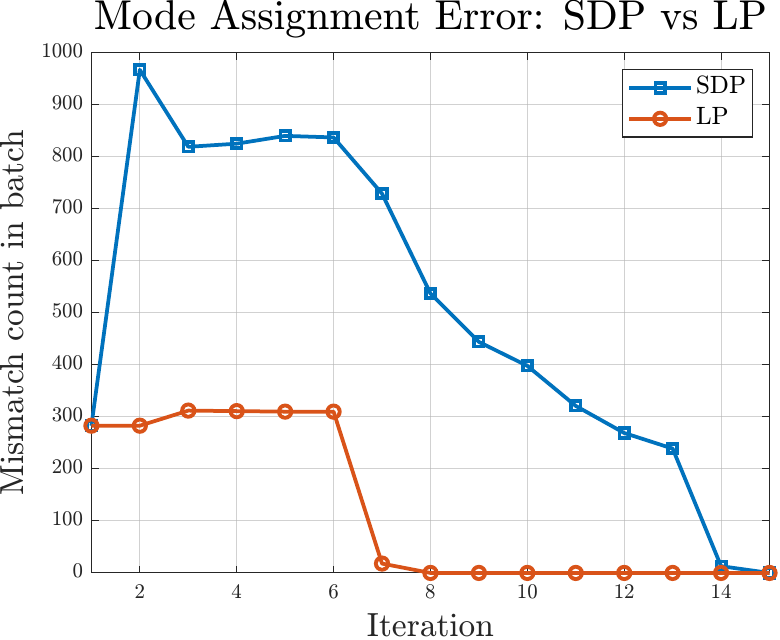} \\
    \includegraphics[width=.5\linewidth]{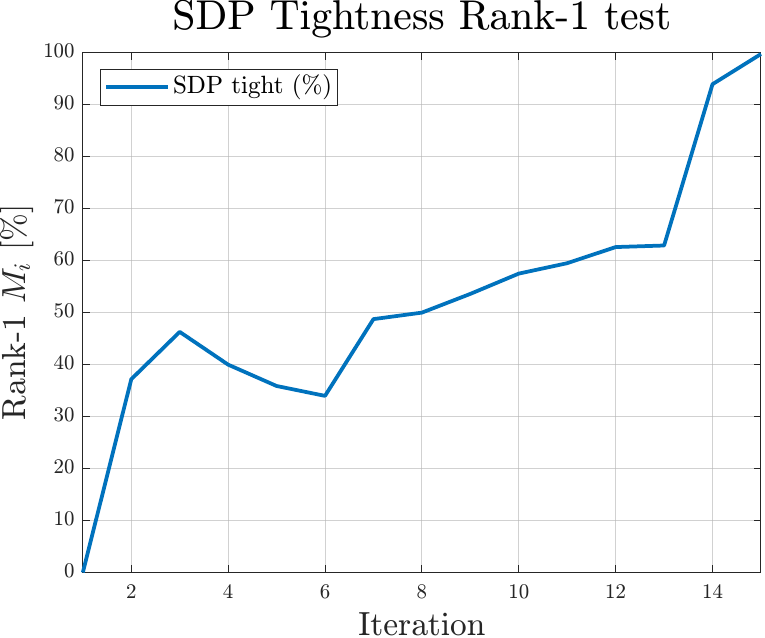} &
    \includegraphics[width=.45\linewidth]{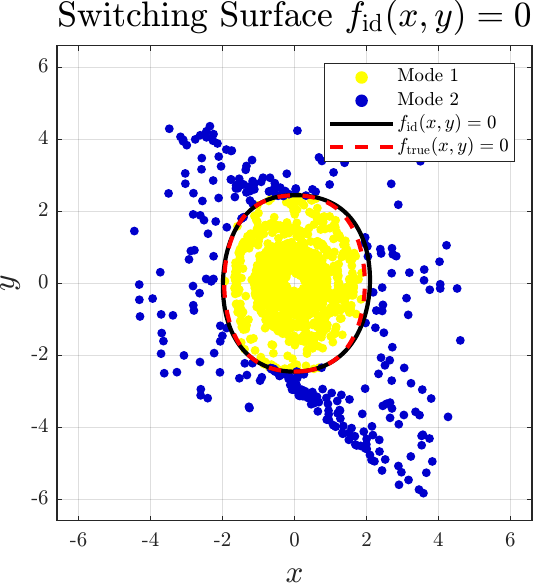}
    \end{tabular}
    \caption{Comparison of convex relaxations for switching polynomial system identification. Top left: identification cost for moment and simplex relaxations. Top right: mode assignment mismatch. Bottom left: SDP rank-1 tightness ratio across relaxations. Bottom right: comparison of the identified switching surface and the ground truth.}
    \label{fig: SPS LP SDP comparison}
    \end{figure}
    
    Over $15$ iterations, the moment relaxation required a total of $1168.608 s$ (median $76.313 s$/iteration), whereas the simplex relaxation required $275.296 s$ in total (median $16.995 s$/iteration).
    As we did in the SLS, we can recover the switching surface from the identified modes. We chose the parameters as $d = 4$, $\eta = 10$, $\varepsilon = 10^{-2}$, and $\beta = 10^{-2}$. The fitted polynomial switching surface
    $f_{\text{identified}}^{-1}(0)$ aligns well with the ground truth surface $f_{\text{true}}^{-1}(0) = f^{-1}(0)$ as can be seen in Figure \ref{fig: SPS LP SDP comparison}. The same set of testing evaluations for this SPS case is summarized in Figure \ref{fig: SPS LP SDP rollouts RMSE} and Table \ref{tab: SPS_metrics}.
    \begin{figure}[!ht]
    \centering
    \small
    \setlength{\tabcolsep}{0pt}
    \begin{tabular}{@{}cc@{}}
    \includegraphics[width=.47\linewidth]{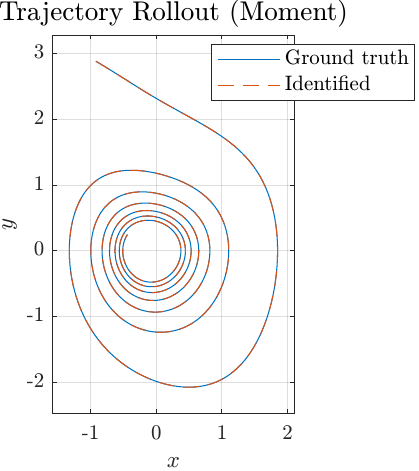} &
    \includegraphics[width=.47\linewidth]{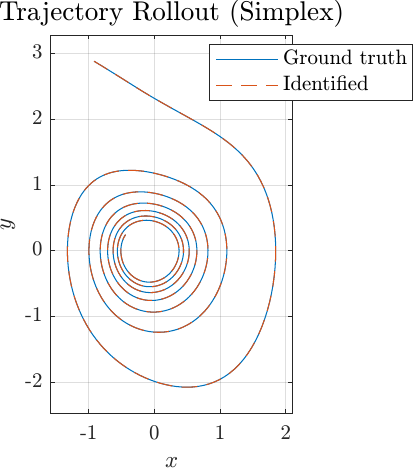} \\
    \includegraphics[width=.47\linewidth]{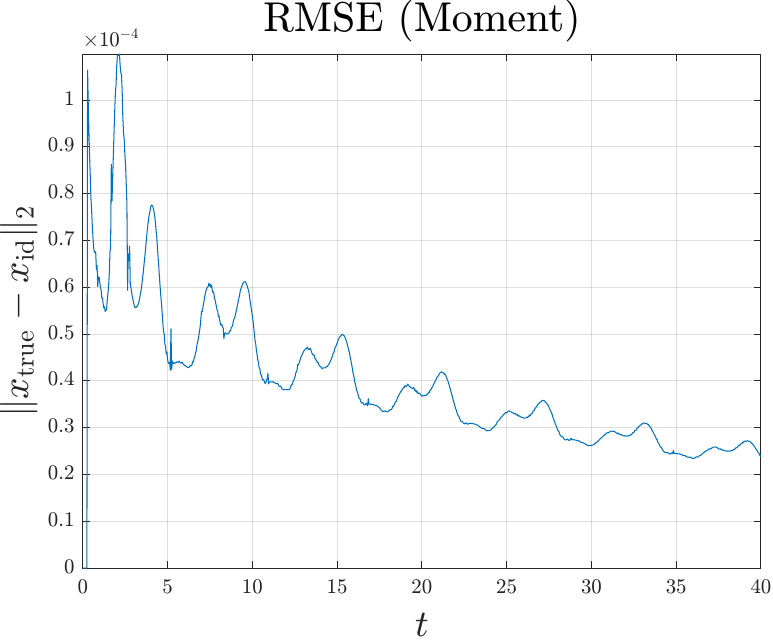} &
    \includegraphics[width=.47\linewidth]{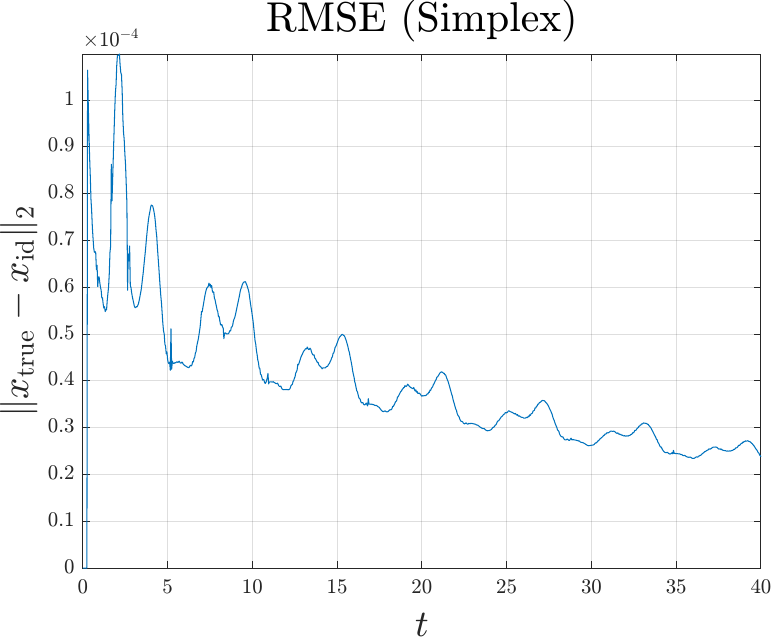}
    \end{tabular}
    \caption{Representative trajectory rollouts (top) and error evolutions (bottom) for the identified switching polynomial system.}
    \label{fig: SPS LP SDP rollouts RMSE}
    \end{figure}
    \begin{table}[!ht]
      \centering
      \renewcommand{\arraystretch}{0.8}
      {%
        \small
        \begin{tabular}{lrr}
          \toprule
          \textbf{Metric (pointwise test)} & \textbf{LP} & \textbf{SDP} \\
          \midrule
          Velocity RMSE & 0.0225 & 0.0225 \\
          Mode Accuracy & 0.9982 & 0.9982 \\
          mIoU          & 0.9929 & 0.9929 \\
          \midrule
          \textbf{Metric (trajectory rollout)} & & \\
          \midrule
          Rollout RMSE  & $0.0060$ & $0.0060$ \\
          Final error   & $0.0022$ & $0.0022$ \\
          Max error     & $0.0206$ & $0.0206$ \\
          \bottomrule
        \end{tabular}%
      }
    \caption{Quantitative evaluation of the identified switching polynomial system.}
    \label{tab: SPS_metrics}
    \end{table}
\section{Conclusions and future work}
In this work, we develop a method for switching system identification using convex relaxation of the original mixed integer program. One limitation is that due to the offline search of switching surfaces, it is hard to identify systems with time-varying or mode-dependent guard sets. Therefore, in future work, we aim to extend this work to the online search as can be seen in \cite{mavridis2024real} for a more general class of switching and hybrid system identification.




\section*{Acknowledgements} Supported in part by NSF grant 2103026, and AFOSR
grants FA9550-32-1-0215 and FA9550-23-1-0400 (MURI).


{\footnotesize 
\bibliographystyle{IEEEtran}
\bibliography{bib/strings-abrv,bib/ieee-abrv,bib/references}
}

\end{document}